\documentclass[11pt,twoside, final]{amsart}
\copyrightinfo{0}{Iranian Mathematical Society}
\pagespan{1}{\pageref*{LastPage}}
\usepackage{etoolbox,lastpage}
\commby{}
\date{\scriptsize   Received: , Accepted: .}
\usepackage{amsmath,amsthm,amscd,amsfonts,amssymb,enumerate}
\usepackage{graphicx}		
\usepackage{color}
\usepackage[colorlinks]{hyperref}

\newtheorem{theorem}{Theorem}[section]
\newtheorem{proposition}[theorem]{Proposition}
\newtheorem{lemma}[theorem]{Lemma}
\newtheorem{corollary}[theorem]{Corollary}
\theoremstyle{definition}
\newtheorem{definition}[theorem]{Definition}
\newtheorem{example}[theorem]{Example}
\theoremstyle{remark}
\newtheorem{remark}[theorem]{Remark}
\numberwithin{equation}{section}

\begin{document}


\title[Topological Entropy and Recurrence  Properties in NDS]{Topological Entropy and Recurrence  Properties in Non-autonomous Dynamical Systems}

\author[Mehdi Fatehi Nia]{ {{\bf {Mehdi Fatehi Nia}}
\\{\tiny{Department of Mathematics, Yazd University, 89195-741 Yazd, Iran}
\\e-mail: fatehiniam@yazd.ac.ir }}}


%

 \maketitle
%

\begin{abstract}
In this paper we study topological entropy and recurrence properties of non-autonomous dynamical system generated by a family of continuous self maps on a compact space $X$. Specially, we introduce the pseudo-entropy and periodic-pseudo-entropy   and prove their equivalence with the topological entropy for non-autonomous dynamical systems.
\end{abstract}

\section{\bf Introduction}
The notion of topological entropy is a well known tool for
measure the complexity of dynamical systems. Topological entropy was introduced
by Adler et al (\cite{[AM]}) and later extended
by Bowen (\cite{[BA]}). After these works many
 research articles appeared on different dynamical systems
and computation methodology of topological entropy. In \cite{[BS]}, the authors prove that the topological entropy of a map is equal to the exponential growth rate of the number of separated (periodic) pseudo orbits. In \cite{[RW]}, the authors introduced the notions of chain mixing  rate and making use of this notion gave a lower bound for topological entropy. \\
In the recent years, non-autonomous discrete dynamical systems
have been extensively studied by many researchers \cite{[CJ], [F],[FN],[KMS],[TD],[Y]}. The main difference between autonomous and non-autonomous systems is that the
basic elements dictating the dynamics are fixed in the former, and change over time in the
latter. 
Specially,  Kolyada and Snoha \cite{[KMS]} introduced topological
entropy for a non-autonomous dynamical system given by a sequence of continuous self-maps of a compact metric space. Recently, Kawan introduced the notion of metric entropy for a non-autonomous dynamical system which   is related via a variational inequality
to the topological entropy of non-autonomous systems as
defined by Kolyada and Snoha. In this way they   generalized several properties of the
classical metric entropy such as  Rokhlin inequality and power rule  to non-autonomous dynamical systems \cite{[CK]}. In \cite{[CJ],[FA],[MO],[TDB]}, the authors  have studied main notions in discrete dynamical systems such as; non-wandering sets, shadowing, chain recurrent,  topological stability,  and  expansiveness  for non-autonomous discrete dynamical systems induced by a sequence of continuous
on a compact metric space.\\
Chain recurrent sets and non-wandering sets have an important role in the study of ergodic properties
of dynamical systems.  In \cite{[TD]}, Thakkar and Das consider the chain recurrent sets of non-autonomous dynamical systems and study chain recurrent sets in a non-autonomous discrete system with the shadowing property. Also, in \cite{[TDA]}  they defined and studied non-wandering set, $\alpha-$limit set, $\omega-$limit
set and recurrent set for non-autonomous discrete dynamical systems.\\
We use briefly \emph{NDS} to denote the non-autonomous dynamical systems.\\
\\In this paper, first we present some definitions and resulsts
that will be used in the sequel. In Section $3,$ we review the work of Thakkar and Das on non-wandering sets and recurrent points for  \textit{NDS}'s and prove that topological entropy of an equi-continuous \textit{NDS} is equal to the topological entropy of this \textit{NDS} restricted to its chain recurrent set. In the rest, shadowing property for a \textit{NDS} considered and this is proved that for an equi-continuous \textit{NDS} with  the shadowing property every recurrent point is a chain recurrent point. Section $4$ is the main part of the paper. There, we introduce pseudo-entropy and periodic-pseudo entropy for non-autonomous dynamical systems and prove the equivalence of these notions and topological entropy for \textit{NDS}'s. This section is a generalization of the results presented in \cite{[BS]}. Finally, in the last section, we study chain mixing and topological mixing in  \textit{NDS}'s and  a lower boundary for topological entropy is  evaluated.
\section{Preliminaries}
In this section, we recall the notion of entropy for a  \textit{NDS}, as defined in Kolyada et al. \cite{[KMS]}.
As in the classical autonomous dynamical systems, the  definition of topological entropy using open covers and the definition using separated/spanning sets are coincide.\\
Let $X$ be a compact topological space and $\mathcal{F}=\{f_{i}\}_{i=1}^{\infty}$  a sequence of
continuous maps from $X$ to $X$. For any positive integers, $i$, $n$; set $\mathcal{F}_{[i,n]}=f_{i+(n-1)} o~f_{i+(n-2)}o~....o~f_{i+1} o~f_{i}$ and additionally $\mathcal{F}_{[i,0]}=id$. We also write $\mathcal{F}_{[i,-n]}=(\mathcal{F}_{[i,n]})^{-1}$ \cite{[KMS]}. For case $i=1$ we will use $\mathcal{F}_{n}$. The pair $(X,\mathcal{F})$ is called a \emph{non-autonomous
discrete dynamical system}. The trajectory of a point $x\in X$ is the sequence $(\mathcal{F}_{n}(x))_{n\geq 0}$.\\
Now we consider the \emph{topological entropy} for a non-autonomous dynamical system.\\ Let $(X,d)$ be a compact metric space. For each $n\geq 1$ the function $d_{n}(x,y)=max\{d(\mathcal{F}_{j}(x),\mathcal{F}_{j}(y)): 0\leq j\leq n-1\}$ is a metric on $X$ and equivalent to $d$.\\ A subset $E$ of $X$ is called $(n,\mathcal{F},\epsilon)-$separated if for any two distinct point $x,y\in E$, $d_{n}(x,y)>\epsilon$. A subset $F$ of $X$ is called $(n, \mathcal{F},\epsilon)-$spanning if for each $x\in X$ there is $y\in F$ for which $d_{n}(x,y)\leq \epsilon$ \cite{[KMS]}.\\
We define $s_n( \mathcal{F},\epsilon)$  as the maximal cardinality of a $(n, \mathcal{F},\epsilon)-$seperated
set and $r_n( \mathcal{F},\epsilon)$  as the minimal cardinality of a set which is a
$(n,\mathcal{F},\epsilon)-$spanning set.
 The topological entropy $h(\mathcal{F})$ of the system $(X,\mathcal{F})$ is defined by $$h(\mathcal{F})=\lim_{\epsilon\rightarrow 0}\lim_{n\rightarrow \infty}\frac{1}{n}\log s_n( \mathcal{F},\epsilon)=\lim_{\epsilon\rightarrow 0}\lim_{n\rightarrow \infty}\frac{1}{n}\log r_n( \mathcal{F},\epsilon).$$
Now, we consider the other definition for topological entropy of non-autonomous dynamical systems which is based on open covers.\cite{[KMS]}\\
For open covers $\mathcal{A}_{1},\mathcal{A}_{2},\cdots, \mathcal{A}_{n}$ we denote:\\
$\displaystyle\bigvee_{i=1}^{n}\mathcal{A}_{i}=\mathcal{A}_{1}\vee\mathcal{A}_{2}\vee\cdots\vee\mathcal{A}_{n}=\{A_{1}\cap A_{2}\cap \cdots\cap A_{n}:A_{1}\in\mathcal{A}_{1},A_{2}\in\mathcal{A}_{2},\cdots A_{n}\in\mathcal{A}_{n}\}.$\\
This is clear that $\displaystyle\bigvee_{i=1}^{n}\mathcal{A}_{i}$ is also an open cover for $X$. For an open cover $ \mathcal{A}$, let $\mathcal{F}_{[i,-n]}(\mathcal{A})=\{\mathcal{F}_{[i,-n]}(A):A\in \mathcal{A}\}$ and $\mathcal{A}_{i}^{n}=\displaystyle\bigvee_{j=0}^{n-1}\mathcal{F}_{[i,-j]}(\mathcal{A})$. Let $\mathcal{N}(\mathcal{A})$ denote the minimal possible cardinality of a subcover chosen from $\mathcal{A}$. Then $$h^{*}(\mathcal{F},\mathcal{A})=\displaystyle\limsup_{n\rightarrow \infty}\frac{1}{n}\log\mathcal{N}(\mathcal{A}_{1}^{n})$$
is said to be the topological entropy of $\mathcal{F}$ on the open cover $\mathcal{A}$.\\
Let $$h^{*}(\mathcal{F})=\sup\{h^{*}(\mathcal{F},\mathcal{A}): \mathcal{A} ~is ~an ~open~cover ~of ~X\}.$$
By Lemma 3.1 of \cite{[KMS]}, $h^{*}(\mathcal{F})=h(\mathcal{F}).$\\
This is well known  that entropy is an important invariant of topological
conjugate, In \cite{[CK]} the authors introduce conjugated non-autonomous systems and prove that topological entropy for non-autonomous dynamical systems is an invariant of topological conjugacy.
\begin{definition}\label{dem}\cite{[CK]}
Let $(X,\mathcal{F})$ and $(Y,\mathcal{G})$ be two non-autonomous system. We say that $(X,\mathcal{F})$ and $(Y,\mathcal{G})$ are conjugated, if for every $n\in N$ there exists a homoeomorphism map $\pi_{n}:X\longrightarrow Y$ such that $\pi_{n+1}of_{n}=g_{n}o\pi_{n}$. If the maps $\pi_{n}$ are only continuous then we say that the systems $(X,\mathcal{F})$ and $(Y,g_{1, \infty})$ are semi-conjugate.
\end{definition}
\begin{proposition} \cite{[CK]}\label{pa}
Let $(X,\mathcal{F})$ be semi-conjugated to $(Y,\mathcal{G})$. Then $h(\mathcal{F})\leq h(\mathcal{G})$.
\end{proposition}
\section{Non-wandering and chain recurrent sets}
Dynamical systems with some kind of orbit's recurrence have been always attractive to researchers and various notions of recurrence have been studied in last years.  In this section we consider the main ones: non-wandering sets, periodic points, chain recurrent sets and recurrent sets for non-autonomous dynamical systems.
\begin{definition}\cite{[TDA]}\label{deb}
Let $(X,d)$ be a metric space and $f_{k}:X\longrightarrow X$ be a sequence of homeomorphisms, $k=1,2,\cdots$. A point $x\in X$ is said to be a non-wandering point for $(X,\mathcal{F})$ if for any neighborhood $U$ of $X$ and for any $n\geq 0$ there exists $m\geq n$ and $r\geq 0$ such that $\mathcal{F}_{[m,r]}(U)\cap U\neq \emptyset$. The set of all non-wandering points is denoted by $\Omega(\mathcal{F})$.
\end{definition}
Theorems 2.1 and 2.2 in \cite{[TDA]} show that if $X$ be compact then $\Omega(\mathcal{F})$ is a nonempty and closed set.
\begin{definition}\cite{[TDA]}\label{dec}
Let $(X, d)$ be a metric space and $f_n : X \longrightarrow X$ be a sequence
of homeomorphisms, $n = 0, 1, 2, ....$ A point $x_0\in X$ is said to be periodic point
of \textit{NDS} $ (X,\mathcal{F})$
 if the orbit of $x_{0}$ is periodic, i.e.
there exists an integer $k > 0$ such that $\mathcal{F}_{ik+j}(x_{0}) = \mathcal{F}_{j}(x_{0})$, for every $i \in \mathbb{N}$ and
$0 \leq j < k$. The set of all periodic points of $\mathcal{F}$ is denoted by $Per(\mathcal{F})$.
\end{definition}
Similarly, a point $x\in X$ is said to be a fixed point of $ (X,\mathcal{F})$
if $f_{n}(x)=x$ for all $n\geq 0$.\\
In \cite{[TDA]}, the authors prove that $ Per(\mathcal{F})\subset\Omega(\mathcal{F})$.
\begin{definition} \cite{[TDA]}\label{ded}
Let $(X,\mathcal{F})$ be a \textit{NDS}.  By $\alpha-$limit set of a point $x\in X$, we mean the
set $$\alpha(x) = \{y \in X| \lim_{k\rightarrow \infty}
d(\mathcal{F}_{n_{k}}(x), y) = 0\},$$
where $\{n_k\}$ is some strictly decreasing sequence of negative integers.\\
Similarly, by $\omega-$limit set of a point $x\in X$, we mean the
set $$\omega(x) = \{y \in X| \lim_{k\rightarrow \infty}
d(\mathcal{F}_{n_{k}}(x), y) = 0\},$$
where $\{n_k\}$ is some strictly increasing sequence of positive integers.
\end{definition}
A point $x \in X$  is said to be recurrent if
$x\in \alpha(x)\cap \omega(x)$. We denote the set of all recurrent points of $\mathcal{F}$ by $R(\mathcal{F})$ and the closure of it by $C(\mathcal{F})$ \cite{[TDA]}.
\begin{remark}\label{rea}
By Theorem 2.5 of \cite{[TDA]}, if $X$ is compact then for any $x\in X$, $\alpha(x)\subseteq \Omega(\mathcal{F})$ and $\omega(x)\subseteq \Omega(\mathcal{F})$. Then
we have $Per(\mathcal{F})\subset R(\mathcal{F})\subseteq \Omega(\mathcal{F})$ and $C(\mathcal{F})\subseteq\Omega(\mathcal{F})$.
\end{remark}
\begin{definition} \cite{[TD]}\label{def}
 Let $ (X,\mathcal{F})$ be a \textit{NDS}. A point
$x \in X$ is said to be a chain recurrent point for $\mathcal{F}$ if for any $\delta > 0$ and any $n \geq 0$, there exist $m \geq n$ and a finite
sequence $\{x_i\}_{i=0}^{k}$
of points of $X$ with $x_0 = x_k = x$ such that $d(f_{m+i}(x_i), x_{i+1}) < \delta$ or $d(f^{-1}
_{m+i}
(x_i), x_{i+1}) < \delta$ for all
$i = 0, 1, . . . , k - 1$. The sequence $\{x_i\}_{i=0}^{k}$
is said to be a $\delta-$chain for $x$ with action starting at $m$. The set of all
chain recurrent points of $\mathcal{F}$ is denoted by $CR(\mathcal{F})$.
\end{definition}
Recall that  the sequence  $\{h_{n}\}_{n\geq 1}$ of homeomorphisms on $X$ is said to be \emph{equi-continuous}, if for any $\epsilon>0$ there exists a constant $\delta>0$ such that such that $d(h_{n}(x),h_{n}(y))<\epsilon$ for all $x,y \in X$ with $d(x,y)<\delta$ and for all $n\geq 1$ \cite{[SA]}.\\
 So we have the following theorem:
\begin{theorem}\label{ta}
Let  $\mathcal{F}$ be an equi-continuous \textit{NDS} then $h(\mathcal{F})=h(\mathcal{F},CR(\mathcal{F})).$
\end{theorem}
\begin{proof}
By Theorems 3.3 and 3.4 of \cite{[TD]}, if the family of homeomorphisms $\{f_{n},~f_{n}^{-1}\}_{n\geq 0}$ is equi-continuous on $X$, then $CR(\mathcal{F})$ is a closed set and $\Omega(\mathcal{F})\subset CR(\mathcal{F})$.
In \cite{[KMS]} the authors introduce the entropy $h(\mathcal{F},Y)$ with respect to any subset $Y$ of $X$ and also prove that $h(\mathcal{F})=h(\mathcal{F},\Omega(\mathcal{F}))$. So $h(\mathcal{F})=h(\mathcal{F},\Omega(\mathcal{F}))\leq h(\mathcal{F},CR(\mathcal{F}))\leq h(\mathcal{F}),$ which complete the proof.
\end{proof}
\begin{definition} \cite{[TDB]}\label{deg}
Let $F = (X,\mathcal{F})$ be a \textit{NDS}. For $\delta > 0$,  the sequence $\{x_n\}_{n=-\infty}^{\infty}$ in $X$ is said to be a $\delta-$pseudo orbit of $\mathcal{F}$ if
$d(f_{n}(x_{n}), x_{n+1})<\delta$ for $n\geq 0$ and $d(f^{-1}_{-n}(x_{n+1}), x_{n})<\delta$ for $n\leq -1$.
For given $\epsilon > 0$, a $\delta-$pseudo orbit $\{x_n\}_{n=-\infty}^{\infty}$ is said to be
$\epsilon-$traced by $y\in X$ if $d (\mathcal{F}_{n}(y),x_{n}) < \epsilon$ for all $n\in  Z.$
The \textit{NDS} $(X,\mathcal{F})$ is said to have
shadowing property or pseudo orbit tracing property (P.O.T.P.)
if, for every $\epsilon > 0$, there exists  $\delta > 0$ such that every $\delta-$pseudo orbit is $\epsilon-$traced by some points of $X$.
\end{definition}
\begin{theorem}
Let $(X,\mathcal{F})$ has the shadowing property and the family homeomorphisms $\{f_{n},~f_{n}^{-1}\}_{n\geq 0}$ is equi-continuous on $X$, then every recurrent point for $\mathcal{F}$ is a chain recurrent point for $\mathcal{F}$.
\end{theorem}
\begin{proof}
In \cite{[TD]}, the authors prove that if $(X,\mathcal{F})$ has the shadowing property then $CR(\mathcal{F})\subset\Omega(\mathcal{F}) $, and consequently, if $(X,\mathcal{F})$ has the shadowing property and the family homeomorphisms $\{f_{n},~f_{n}^{-1}\}_{n\geq 0}$ is equi-continuous on $X$, then $CR(\mathcal{F})=\Omega(\mathcal{F})$ and hence Remark \ref{rea} implies that $C(\mathcal{F})\subset CR(\mathcal{F})$.
\end{proof}
\section{Pseudo orbits and entropy}
In \cite{[BS]} Barge and Swanson use of pseudo-orbits and periodic-pseudo-orbits instead of orbits in their definition of topological entropy which is one of the equivalent definitions of topological entropy. In this section, as the main part of the paper,  \emph{pseudo-entropy}   and \emph{periodic-pseudo-entropy} for non-autonomous dynamical systems introduced and the equivalence of that  with topological entropy on non-autonomous dynamical systems investigated.\\
A subset $E$ of $(\alpha,\mathcal{F})$-pseudo orbits is $(n,\epsilon)-$separated if, for each distinct sequences $\texttt{x}=\{x_n\}_{n=-\infty}^{\infty}$, $\texttt{y}=\{y_n\}_{n=-\infty}^{\infty}$ in $E$, there is a  $0\leq k< n$, for which $d(x_{k},y_{k})>\epsilon$. Let $c_{n}(\epsilon,\alpha,\mathcal{F})$ denote the maximal cardinality of a $(n,\epsilon)-$separated set of $(\alpha,F)$-pseudo orbits. Since $X$ is compact then $c_{n}(\epsilon,\alpha,\mathcal{F})$ is finite.
\begin{definition}\label{deh}
The number $h_{p}(\mathcal{F})=\displaystyle \lim_{\epsilon\rightarrow 0}\lim_{\alpha\rightarrow 0}\displaystyle\limsup_{n\rightarrow \infty}\frac{1}{n}\log c_{n}(\epsilon,\alpha,\mathcal{F})$ will be called the pseudo-entropy of $\mathcal{F}$.
\end{definition}
Now, we are going to show that $h_{p}(\mathcal{F})=h(\mathcal{F})$. For this purpose we need some notions and lemmas. Our proof of Lemmas and Theorems in this section  is based upon ideas found in\cite{[BS]}.\\
Suppose that $X_{\alpha}$ is the set of all $(\alpha,\mathcal{F})$-pseudo orbits and define the metric $\rho$ on $X_{\alpha}$ by $\rho(\texttt{x},\texttt{y})=\displaystyle\Sigma_{i=-\infty}^{\infty}\frac{d(x_{i},y_{i})}{2^{|i|}}$. By definition, $X_{\alpha}$ is a closed subset of $X^{Z}$ and hence is compact. Let $\sigma_{\alpha}:X_{\alpha}\rightarrow X_{\alpha}$ denote the shift
$$\sigma_{\alpha}(\cdots,x_{-1};x_{0},x_{1},\cdots)=(\cdots,x_{-1},x_{0};x_{1},x_{2},\cdots).$$
For each $\epsilon>0$, let $\emph{A}_{1,\epsilon}$ be a finite cover of $X_{1}$ by $\epsilon-$balls and form the restricted covers $\emph{A}_{\alpha,\epsilon}=\{A\cap X_{\alpha}/ A\in \emph{A}_{1,\epsilon}\}$, $0\leq \alpha\leq 1$ \cite{[BS]}.
\begin{lemma}\label{lec}
$h(\sigma_{0})=h(\mathcal{F})$.
\end{lemma}
\begin{proof}
Consider the \textit{NDS} $G=(X_{0},\{g_{i}\})$, where $g_{i}=\sigma_{0}$, for all $i\geq 0$. Let $\pi_{k}:X_{0}\rightarrow X$ be given by $\pi_{k}(\texttt{x})=x_{k}$, $k\in Z$. So $f_{k}\circ \pi_{k}=\pi_{k+1}\circ\sigma_{0}$, where for $k<0$ we define $f_{k}(x)=f^{-1}_{-k}(x)$. Then $\mathcal{F}$ is semi conjugate to $G$ and hence by Proposition \ref{pa}, $h(\mathcal{F})\leq h(\mathcal{G})=h(\sigma_{0})$. On the other hand by Lemma 1 of \cite{[BS]} and Proposition 3.2 of \cite{[CK]}, for any finite cover $\mathcal{A}$ of $X_{0}$ there is a finite cover $\mathcal{B}$ of $X$ such that $h(\sigma_{0},\mathcal{A})\leq h(\mathcal{F},\mathcal{B})$. So that $h(\sigma_{0})\leq h(\mathcal{F})$.
\end{proof}
\begin{lemma}\label{led}\cite{[BS]}
$h(\sigma_{0},\mathcal{A}_{0,\epsilon})\geq \displaystyle\inf_{0<\alpha\leq 1}h(\sigma_{\alpha},\mathcal{A}_{\alpha,\epsilon})$
\end{lemma}
\begin{lemma}\label{lef}
$\displaystyle\limsup_{n\rightarrow \infty}\frac{1}{n}\log~ c_{n}(2\epsilon,\alpha,\mathcal{F})\leq h(\sigma_{\alpha},\mathcal{A}_{\alpha,\epsilon}).$
\end{lemma}
\begin{proof}
This is clear that for any pair $(n,\epsilon)-$separated, $\texttt{x}=\{x_n\}_{n=-\infty}^{\infty}$ and
$\texttt{y}=\{y_n\}_{n=-\infty}^{\infty}$ of $(\alpha,\mathcal{F})$-pseudo orbits in  $E$, $\texttt{x}$ and $\texttt{y}$ are $(n,\epsilon)-$separated orbits of
$\sigma_{\alpha}:X_{\alpha}\rightarrow X_{\alpha}$. Then, by proof of Lemma 3 of \cite{[BS]}, the maximal cardinality of a
$(n,\epsilon)-$separated set for $\mathcal{F}$ is not greater than $card(\bigvee_{i=0}^{n-1}\sigma_{\alpha}^{-1}(\mathcal{A}_{\alpha,\epsilon}))$.
 So $\displaystyle\limsup_{n\rightarrow \infty}\frac{1}{n}\log c_{n}(\epsilon,\alpha,\mathcal{F})\leq h(\sigma_{\alpha},\mathcal{A}_{\alpha,\epsilon}).$
\end{proof}
The following theorem is one the main result of this paper.
\begin{theorem}\label{tb}
Let $X$ be a compact metric space and $(X,\mathcal{F})$ be a non-autonomous dynamical systems on $X$. Then $h_{p}(\mathcal{F})=h(\mathcal{F})$.
\end{theorem}
\begin{proof}
Since every orbit of $\mathcal{F}$ is an $(\alpha,\mathcal{F})$-pseudo orbit, for all $\alpha>0$, this is clear that $h(\mathcal{F})\leq h_{p}(\mathcal{F})$. On the other hand, Lemma \ref{lef} implies that
$\displaystyle \lim_{\alpha\rightarrow 0}\displaystyle\limsup_{n\rightarrow \infty}\frac{1}{n}\log c_{n}(2\epsilon,\alpha,\mathcal{F})\leq \displaystyle\inf_{0<\alpha \leq 1} h(\sigma_{\alpha},\mathcal{A}_{\alpha,\epsilon}).$ Then by Lemma \ref{led}\\ $\displaystyle \lim_{\alpha\rightarrow 0}\displaystyle\limsup_{n\rightarrow \infty}\frac{1}{n}\log c_{n}(2\epsilon,\alpha,\mathcal{F})\leq  h(\sigma_{0},\mathcal{A}_{0,\epsilon}).$ So, Letting $\epsilon\rightarrow 0$ and use of Lemma \ref{lec}, we have  $h_{p}(\mathcal{F})\leq h(\mathcal{F})$.
\end{proof}
In \cite{[KMS]} the authors cosider the relation between the topological entropy of a \textit{NDS}, as  a uniformly convergent sequence of
maps and the classical topological entropy of its limit and gave the following lemma.
\begin{lemma}\label{leb}\cite{[KMS]}
Let $\mathcal{F}$ be a sequence of continuous self-maps of a
compact metric space $X$ converging uniformly to $f$. Then $h(\mathcal{F})\leq h(f)$.
\end{lemma}
By Lemma \ref{leb} and Theorem \ref{tb} we have the following result:
\begin{corollary}
Let $X$ be a compact metric space and $(X,\mathcal{F})$ be a \textit{NDS}  contains a sequence of continuous functions converging uniformly to $f$. Then $h_{p}(\mathcal{F})\leq h(f)$.
\end{corollary}
\subsection{Chain transitivity and periodic-pseudo entropy}
\begin{definition}\label{deg}
We say that the \textit{NDS} $(X,\mathcal{F})$  is $\alpha$-\emph{chain transitive} if for every $x,y\in X$ there is an $(\alpha,\mathcal{F})$-chain from $x$ to $y$ and an $(\alpha,\mathcal{F})$-chain from $y$ to $x$. The \textit{NDS} $(X,\mathcal{F})$  is chain transitive if for every $\alpha>0$, is an $\alpha$-chain transitive.
\end{definition}
\begin{lemma}\label{leg}
Let  the \textit{NDS} $(X,\mathcal{F})$  be chain transitive, then there is a positive number $K$ such that for every pair $x,y\in X$ there is a $(2\alpha,\mathcal{F})$-chain from $x$ to $y$ of length less than or equal to $K$.
\end{lemma}
\begin{proof}
For every $(a,b)\in X\times X$, choose the number $k(a,b)$ such that there is a $(2\alpha,F)$-chain from $a$ to $b$. So, we can find open sets $V_{a},~V_{b}$  contains $a$ and $b$, respectively, such that if $a_{1}\in V_{a}$ and $b_{1}\in V_{b}$ then there is a $(2\alpha,\mathcal{F})$-chain of length $k(a,b) $ from $a_{1}$ to $b_{1}$. The collection $\{V_{a}\times V_{b}\}_{a,b\in X}$ is an open cover for $X\times X$. Compactness of $X$ implies that  the open cover $\{V_{a}\times V_{b}\}_{a,b\in X}$  for $X\times X$ has a finite subcover $\{V_{a_{i}}\times V_{b_{i}}\}$. Let $K=max\{k(a_{i},b_{i})\}$.
\end{proof}
A subset $S$ of $(\alpha,\mathcal{F})$-pseudo orbits of period $n$ is $(n,\epsilon)-$separated if, for each $\texttt{x}=\{x_i\}_{0\leq i< n}$, $\texttt{y}=\{y_i\}_{0\leq i< n}\in S$, $\texttt{x}\neq\texttt{y}$, there is a $i$, $0\leq i< n$, for which $d(x_{i},y_{i})>\epsilon$. Let $p_{n}(\epsilon,\alpha,\mathcal{F})$ denote the maximal cardinality of a $(n,\epsilon)-$separated set of $(\alpha,F)$-pseudo orbits. Since $X$ is compact then $p_{n}(\epsilon,\alpha,\mathcal{F})$ is finite.
\begin{remark}\label{reg}
Let $\alpha^{'}<\alpha$, since every $(\alpha^{'},\mathcal{F})$-pseudo orbit is an $(\alpha,\mathcal{F})$-pseudo orbit, then  $p_{n}(\epsilon,\alpha,\mathcal{F})\leq p_{n}(\epsilon,\alpha^{'},\mathcal{F})$.
\end{remark}
 The number $H_{p}(\mathcal{F})=\displaystyle \lim_{\epsilon\rightarrow 0}\lim_{\alpha\rightarrow 0}\displaystyle\limsup_{n\rightarrow \infty}\frac{1}{n}\log p_{n}(\epsilon,\alpha,\mathcal{F})$ will be called the periodic-pseudo-entropy of $\mathcal{F}$. The following theorem shows that $H_{p}(\mathcal{F})=h(\mathcal{F})$.

\begin{theorem}\label{tc}
Let  the \textit{NDS} $(X,\mathcal{F})$  be chain transitive. The topological entropy $h(\mathcal{F})$ is equal to $H_{p}(\mathcal{F})$.
\end{theorem}
\begin{proof}
Let $E$ be a $(n,\epsilon)-$separated set of $(\alpha,\mathcal{F})$-pseudo orbits. For every $\texttt{x}=\{x_i\}$ in $E$ we have an $\alpha-$pseudo orbit from $x_{1}$ to $x_{n}$. By Lemma \ref{leg} there is a $2\alpha$-pseudo orbit of the length at most $K$ from $x_{n}$ to $x_{1}$. So there exists a $2\alpha$-periodic-pseudo orbit of the length at most $K+n$ from $x_{1}$ to $x_{1}$. This implies that $c_{n}(\epsilon,\alpha,\mathcal{F})\leq \displaystyle\Sigma_{i=1}^{K+n}p_{i}(\epsilon,\alpha,\mathcal{F}).$ \\
For each $m\geq 1$, let $i_{m}$, $1\leq i_{m}\leq m$, be such that $p_{i_{m}}(\epsilon,2\alpha,\mathcal{F})\geq p_{i}(\epsilon,2\alpha,\mathcal{F})$, for all $1\leq i\leq m$. Then,
\begin{align*}
h(\mathcal{F})&=h_{p}(\mathcal{F})& ~~~~~~~~\text{Theorem\ref{tb}}\\
&\leq \displaystyle\limsup_{n\rightarrow \infty}\frac{1}{n}\log c_{n}(2\epsilon,\alpha,\mathcal{F})&~~~~~~~~~~~~~~ \text{Remark\ref{reg}}
\\ &\leq \displaystyle\limsup_{n\rightarrow \infty}\frac{1}{n}\log\displaystyle\Sigma_{i=1}^{K+n}p_{i}(\epsilon,2\alpha,\mathcal{F})
\\& \leq \displaystyle\limsup_{n\rightarrow \infty}\frac{1}{K+n}\log\displaystyle (K+n)p_{i_{K+n}}(\epsilon,2\alpha,\mathcal{F})
\\ &\leq \displaystyle\limsup_{n\rightarrow \infty}\frac{1}{i_{K+n}}\log\displaystyle (K+n)p_{i_{K+n}}(\epsilon,2\alpha,\mathcal{F})
\\ &\leq  \displaystyle\limsup_{n\rightarrow \infty}\frac{1}{n}\log p_{n}(\epsilon,2\alpha,\mathcal{F}).
\end{align*}
So $h(\mathcal{F})\leq H_{p}(\mathcal{F})$. On the other hand, by definitions this is clear that $H_{p}(\mathcal{F})\leq h_{p}(\mathcal{F})$ and hence by Theorem \ref{tb} $H_{p}(\mathcal{F})\leq h(\mathcal{F})$.
\end{proof}
By Lemma \ref{leb} and Theorem \ref{tc} we have the following result:
\begin{corollary}
Let $X$ be a compact metric space and $(X,\mathcal{F})$ be a chain transitive \textit{NDS}  contains a sequence of continuous functions converging uniformly to $f$. Then $H_{p}(\mathcal{F})\leq h(f)$.
\end{corollary}
\begin{example}\label{exa}
Let $I$ be the unit interval and let $g,~h$ be defined as
\[g(x)= \left\lbrace
  \begin{array}{c l}
    2x+\frac{1}{2} & \text{for $x\in [0,\frac{1}{4}]$},\\
    -2x+\frac{3}{2} & \text{for $x\in [\frac{1}{4},\frac{3}{4}]$},\\
    2x-\frac{3}{2} & \text{for $x\in [\frac{3}{4},1]$}.\\
  \end{array}
\right. \]
\[h(x) = \left\lbrace
  \begin{array}{c l}
    x+\frac{1}{2} & \text{for $x\in [0,\frac{1}{2}]$},\\
     -4x+3 & \text{for $x\in [\frac{1}{2},\frac{3}{4}]$},\\
    2x-\frac{3}{2} & \text{for $x\in [\frac{3}{4},1]$}\\
  \end{array}
\right. \]
Consider the \textit{NDS} $(I,\mathcal{F})$, where $\mathcal{F}=\{g,h,g,h,g,\cdots\}. $ In \cite{[SR]}, the authors prove $(I,\mathcal{F})$ is transitive and hence by Theorem \ref{tc} the topological entropy $h(\mathcal{F})$ is equal to $H_{p}(\mathcal{F})$.
\end{example}
\begin{definition}\label{dek}\cite{[TDB]}
Let $(X,d)$ be a metric space and $f_{n}:X\rightarrow X$
a sequence of homeomorphisms, $n= 0, 1, 2, \cdots$. The \textit{NDS} $(X,\mathcal{F})$  is said to be expansive
if there exists a constant $e > 0$ (called an expansive constant)
such that, for any $x,~y\in X$, $x\neq y$,  $d(\mathcal{F}_{n}(x),\mathcal{F}_{n}(y))>e$ for
some $n\in N$.
\end{definition}
Fix $n\geq1$. Let $Fix(\mathcal{F}^{n})$ be the set of all point $x$ that $\mathcal{F}_{[k,n]}(x)=x$ and $\mathcal{N}(Fix(\mathcal{F}^{n}))$ denote its cardinality.
\begin{theorem}\label{td}
Let  $(X,\mathcal{F})$ be an expansive  chain transitive \textit{NDS}. If $(X,\mathcal{F})$ has the shadowing property then $h(\mathcal{F})= \displaystyle\limsup_{n\rightarrow \infty}\frac{1}{n}\log[\mathcal{N}(Fix(\mathcal{F}^{n}))].$
\end{theorem}
\begin{proof}
Consider $e>0$ as the expansivity constant and let $\epsilon<e$. Since $\mathcal{F}$ has the shadowing property then there exists $\alpha<\epsilon$ such that every $\alpha-$pseudo periodic orbit is $\epsilon-$traced by some point $y$ in $X$. Up to this point and expansivity of $\mathcal{F}$ for each $\alpha-$pseudo periodic there exists a corresponding periodic orbit. Then by Theorem \ref{tc} $h(\mathcal{F})= \displaystyle\limsup_{n\rightarrow \infty}\frac{1}{n}\log[\mathcal{N}(Fix(\mathcal{F}^{n}))]$
\end{proof}
\section{Topological entropy and chain recurrent}
In this section we investigate the structure of topological mixing \textit{NDS}'s and define \emph{chain mixing time} for \textit{NDS}'s. Theorem \ref{prb}  gives lower bound of the topological entropy  in non-autonomous theory by using  the chain mixing times.
\begin{definition}\cite{[SR]}\label{deh}
 Let $X$ be a compact metric space and $(X,\mathcal{F})$ be a \textit{NDS}. The system $\mathcal{F}$ is said to be \emph{topological mixing}  if for every non-empty open sets $U,V$ there exists a natural number $K$ such that $f_{n}\circ f_{n-1}\circ....\circ f_{1}(U)\cap V\neq\emptyset,$ for all $n\geq K$.
 \end{definition}
 \begin{lemma}\label{leh}\cite{[SR]}
 The DNS $F = (X,\mathcal{F})$ is topologically mixing if and only if for each non-empty
open set $U$, $\displaystyle\lim_{n\rightarrow \infty}f_{n}\circ f_{n-1}\circ....\circ f_{1}(U)=X$
\end{lemma}
\begin{corollary}\label{cob}
The DNS $F = (X,\mathcal{F})$ is topologically mixing if and only if for each $\epsilon>0$ there exists $N_{\epsilon}>0$ such that for any $x,y\in X$ and any $n\geq N_{\epsilon}>0$ there is an $\epsilon-$pseudo-orbit from $x$ to $y$ of length exactly $n$.
\end{corollary}
\begin{definition}\label{dei}
If $0<\epsilon<\delta$ and $x\in X$, define the chain mixing time $m_{\epsilon}(x,\delta,\mathcal{F})$ to be the smallest $N$ such that for any $n\geq N$ and any $y\in X$, there is an $\epsilon-$chain of length exactly $n$ from some point in $B_{\delta}(x)$ to $y$. We define $m_{\epsilon}(\delta,\mathcal{F})$ to be the maximum over all $x$ of $m_{\epsilon}(x,\delta,\mathcal{F})$.
\end{definition}
The compactness of $X$ implies the existence of  the number $m_{\epsilon}(\delta,\mathcal{F})$.
\begin{remark}
Let $\epsilon>0$ and $a\in X$. We define $(B_{\epsilon}o f_{n})(a)= \{x\in X: d(f_{n}(a),x)<\epsilon\}$ and for $U\subset X$, $(B_{\epsilon}o f_{n})(U)=\bigcup_{a\in U}(B_{\epsilon}o f_{n})(a)$.\\
 Put $$(B_{\epsilon}o \mathcal{F}_{2})(a)=(B_{\epsilon}o f_{2})((B_{\epsilon}o f_{1}(a))$$and
$$(B_{\epsilon}o \mathcal{F}_{n})(a)=(B_{\epsilon}o f_{n})((B_{\epsilon}o \mathcal{F}_{n-1}(a)))$$
for all $n\geq 1$.\\
So, $m_{\epsilon}(x,\delta,\mathcal{F})$ is the smallest $N$ such that $(B_{\epsilon}o \mathcal{F}_{N})(B_{\delta}(x))=X$.
\end{remark}
\begin{definition}
We say that the \emph{NDS}, $\mathcal{F}$ is Lipschitz with Lipschitz constant $c$, if $d(f_{n}(x),f_{n}(y))\leq c d(x,y)$, for all $n\geq 1$ and $x,y \in X$.
\end{definition}
\begin{theorem}\label{prb}
Let $\mathcal{F}$ be chain mixing and have Lipschitz constant $c$. Let $D$ be the diameter of $X$. Then for $\delta$ sufficiently small, $m_{\epsilon}(\delta,\mathcal{F})\geq log_{c}(\frac{D(c-1)+2\epsilon}{2\delta(c-1)+2\epsilon})$ if $c>1$ and $m_{\epsilon}(\delta,\mathcal{F})\geq \frac{D-2\delta}{2\epsilon}$ if $c=1$.
\end{theorem}
\begin{proof}
By definition  of $B_{\epsilon}o \mathcal{F}_{n}$ this is clear that \\$diam((B_{\epsilon}o f_{1})(B_{\delta}(x))\leq c(2\delta)+2\epsilon.$ So,  $diam((B_{\epsilon}o \mathcal{F}_{2})(B_{\delta}(x))\leq c(c(2\delta)+2\epsilon)+2\epsilon.$ Then by induction $diam((B_{\epsilon}o \mathcal{F}_{n})(B_{\delta}(x))\leq c^{n}(2\delta)+c^{n-1}(2\epsilon)+c^{n-2}(2\epsilon)+...+2\epsilon\leq c^{n}(2\delta)+\frac{1-c^{n}}{1-c}(2\epsilon).$ \\Since $m_{\epsilon}(x,\delta,\mathcal{F})$ is the smallest $N$ such that $(B_{\epsilon}o \mathcal{F}_{n})(B_{\delta}(x))=X$. Then $N$ is at least $log_{c}(\frac{D(c-1)+2\epsilon}{2\delta(c-1)+2\epsilon}).$
\end{proof}
\begin{proposition}\label{prf}
Let $F = (X,\mathcal{F})$ be a topological mixing \textit{NDS}. Then the topological entropy $h(\mathcal{F})$, satisfies
 $$h(\mathcal{F})\geq d^{'}. \displaystyle\lim\sup_{\delta\rightarrow 0}\frac{log(\frac{1}{\delta})}{\displaystyle\lim\sup_{\epsilon\rightarrow 0}m_{\epsilon}(\delta,\mathcal{F})}$$
where $d^{'}$ is the lower box dimension of $X$.
\end{proposition}
\begin{proof}
Consider $c_{n}(\epsilon,\alpha,\mathcal{F})$ as the maximal cardinality of a $(n,\epsilon)-$separated set of $(\alpha,\mathcal{F})$-pseudo orbits. For $\alpha>0$, let $N(\alpha)=c_{0}(0,\alpha,\mathcal{F})$.\\Let $x_{1},\cdots, x_{N(3\delta)}$ be a $3\delta-$ separated set of points. Fixed $k\geq0$, for each sequence $(i_{0},\cdots,i_{k}),$ where $1\geq i_{j}\leq N(3\delta)$, there is a sequence of points $y_{i_{1}},\cdots, y_{i_{k}}$ such that for each $j$, $y_{i_{j}}\in B_{\delta}(x_{i})$ and there is an $\epsilon-$pseudo-orbit of length $m_{\epsilon}(\delta)$ from $y_{i_{j}}$ to $y_{i_{j+1}}$. Since the points $x_{i}$ are $3\delta-$separated, the sequences $(y_{i_{1}},\cdots, y_{i_{k}})$ are $\delta-$separated . Then $c_{km_{\epsilon}(\delta)}(\epsilon,\delta,\mathcal{F})\geq (N(3\delta))^{k+1}$.\\ By the proof of  Theorem 28 in \cite{[RW]}, for small enough $\alpha$ there exists a positive constant $C$ such that $N(\alpha)\geq C(\frac{1}{\alpha})^{d^{'}}$. Then
\begin{align*}
h(\mathcal{F})&=\displaystyle \lim_{\delta\rightarrow 0}\lim_{\epsilon\rightarrow 0}\displaystyle\lim\sup_{n\rightarrow \infty}\frac{1}{n}log c_{n}(\epsilon,\alpha,\mathcal{F})&
 ~~~~~~~~\text{~~~~~~~~~~Theorem~\ref{tb}}\\
&\geq\displaystyle\limsup_{n\rightarrow \infty}\frac{1}{km_{\epsilon}(\delta)}log (N(3\delta))^{k+1}\\&=
 \displaystyle\limsup_{\delta\rightarrow 0}\lim_{\epsilon\rightarrow 0}\frac{\log N(3\delta)}{m_{\epsilon}(\delta,\mathcal{F})}\\&
 \geq \displaystyle\limsup_{\delta\rightarrow 0}\lim_{\epsilon\rightarrow 0}\frac{\log\frac{C}{(3\delta)^{d^{'}}} }{m_{\epsilon}(\delta,\mathcal{F})}\\&
 =d^{'}.\displaystyle\limsup_{\delta\rightarrow 0}\frac{\log\frac{1}{\delta}}{\lim_{\epsilon\rightarrow 0}m_{\epsilon}(\delta,\mathcal{F})}
\end{align*}
\end{proof}


\end{document}